\documentclass[a4paper,12pt]{article}
\usepackage{latexsym,amsfonts,amsmath,theorem,amssymb}
\usepackage{graphicx}
\usepackage{color}
\usepackage{enumerate}
\usepackage{hyperref}

\newtheorem{theorem}{Theorem}
\newtheorem{lemma}[theorem]{Lemma}

\newtheorem{proposition}[theorem]{Proposition}

\newenvironment{proof}{\begin{trivlist}
    \item[\hskip\labelsep{\bf Proof.}]}{$\hfill\Box$\end{trivlist}}

{\theoremstyle{plain} \theorembodyfont{\rmfamily}
\newtheorem{remark}[theorem]{Remark}}
{\theoremstyle{plain} \theorembodyfont{\rmfamily}
\newtheorem{example}[theorem]{Example}}

\allowdisplaybreaks

\newcommand{\icomp}{\mathtt{i}}

\newcommand{\To}{\rightarrow}
\newcommand{\il}{\left\langle}
\newcommand{\ir}{\right\rangle}

\newcommand{\bsx}{{\boldsymbol{x}}}

\newcommand{\bsz}{{\boldsymbol{z}}}

\newcommand{\bsX}{{\boldsymbol{X}}}

\newcommand{\rd}{\mathrm{d}}

\newcommand{\bbR}{\mathbb{R}}

\newcommand{\bbN}{\mathbb{N}}
\newcommand{\bbZ}{\mathbb{Z}}
\newcommand{\bbE}{\mathbb{E}}

\newcommand{\mask}[1]{}

\newcommand{\norm}[1]{\left\Vert#1\right\Vert}
\newcommand{\abs}[1]{\left\vert#1\right\vert}

\newcommand{\EE}{\mathbb{E}}
\newcommand{\ee}{{\rm e}}

\newcommand{\EXCLUDE}[1]{}

\makeatletter
\newcommand{\rdots}{\mathinner{\mkern1mu\lower-1\p@\vbox{\kern7\p@\hbox{.}}
\mkern2mu \raise4\p@\hbox{.}\mkern2mu\raise7\p@\hbox{.}\mkern1mu}}
\makeatother

\date{}

\begin{document}

\title{Truncation in Average and Worst Case Settings\\ for Special
Classes of $\infty$-Variate Functions}

\author{Peter Kritzer\thanks{P. Kritzer is supported by the Austrian
Science Fund (FWF), Project F5506-N26.}, Friedrich Pillichshammer\thanks{F. Pillichshammer is supported by the 
Austrian Science Fund (FWF) Project F5509-N26. Both projects are parts
of the Special Research Program "Quasi-Monte Carlo Methods:
Theory and Applications".}, G.~W.~Wasilkowski}

\maketitle

\begin{abstract}
The paper considers truncation errors for functions of the form
$f(x_1,x_2,\dots)=g(\sum_{j=1}^\infty x_j\,\xi_j)$, i.e.,
errors of approximating $f$ by
$f_k(x_1,\dots,x_k)=g(\sum_{j=1}^kx_j\,\xi_j)$, where the
numbers $\xi_j$ converge to zero sufficiently fast and
$x_j$'s are i.i.d. random variables. As explained in the
introduction, functions $f$ of the form above appear in a number of
important applications. To have positive results for possibly large classes
of such functions, the paper provides sharp bounds on  truncation errors
in both the average and worst case settings. In the former case,
the functions $g$ are from a Hilbert space $G$ endowed with a zero
mean probability measure with a given covariance kernel. In the latter case,
the functions $g$ are from a reproducing kernel Hilbert space, or
a space of functions satisfying a H\"older condition.
\end{abstract}

\centerline{\begin{minipage}[hc]{130mm}{
      {\em Keywords:} Dimension truncation, Average case error, Worst case error, Covariance kernel, 
      Reproducing kernel\\
{\em MSC 2000:}  65D30, 65Y20, 41A55, 41A63}
\end{minipage}}

\section{Introduction}

In this paper, we are interested in problems that require 
computation of the expectation of $g(\bsX(t))$, where $\bsX(t)$ is the 
value at time $t$ of a stochastic process $\bsX$,
and $g$ is a function from a given function space $G$. 

Such a situation may, for example, occur in the context of 
mathematical finance, or when studying PDEs with random coefficients; 
the latter topic has attracted much interest recently in the field of 
quasi-Monte Carlo (QMC) methods. To be more precise, 
the term $g(\bsX(t))$,
for a given and fixed time $t$,
could be a quantity of interest 
obtained from the solution of a PDE in which one of the coefficients 
is modeled as a random field.
We refer to \cite{KN16} for a recent and 
detailed overview. 

Let us in the following assume that $\bsX$ can be expressed in terms of its 
Karhunen-Lo\`{e}ve (cf.~\cite{L78}) expansion, 
\[\bsX(t) = \sum_{j=1}^\infty x_j\, \psi_j(t),
\]
where $(\psi_j)_{j \ge 1}$ form an orthonormal basis and $(x_j)_{j \ge 1}$ are i.i.d. 
random variables with the corresponding probability measure denoted
by $\omega$. In this case, the expectation problem reduces 
to the integration of 
\[f(\bsx) = g\left(\sum_{j=1}^\infty x_j\, \xi_j\right)\quad\mbox{with}\quad
\xi_j = \psi_j(t)
\]
with respect to $\omega^\bbN$, the countable product of $\omega$.

As in \cite{HKPW,KPW}, 
the main focus of the paper is on the truncation errors, i.e., errors 
caused by replacing the infinite sum $\sum_{j=1}^\infty x_j\,\xi_j$ 
with the truncated sum $\sum_{j=1}^k x_j\,\xi_j$.
Here we 
study how the
truncation errors depend on $k$ in the
{\em average case} and {\em worst case} settings with respect to 
functions~$g$.

Throughout this paper we assume that
\[\sum_{j =1}^{\infty} |\xi_j|\,<\, \infty.
\]

\section{Average and Worst Case Settings}
        
We consider two settings: the {\em average} and 
{\em worst case settings}
for spaces $G$ of functions
\[g\,:\,D\to\,\bbR,
\]
where $D$ is an interval (possibly unbounded) in $\bbR$.
In the former setting, 
$G$ is a Hilbert space endowed with 
a zero mean probability measure $\mu$ whose covariance kernel
is denoted by $K_\mu^{\rm cov}$. In the latter setting, 
the space $G$ is either a reproducing kernel Hilbert space whose 
reproducing kernel is denoted by $K^{\rm rep}$, or a normed space
of functions satisfying a H\"older condition.

Recall that the covariance kernel of a measure $\mu$ on $G$ is defined by
\[
K_\mu^{\rm cov}(x,y)\,=\,\bbE_\mu(g(x)\,g(y))\,=\,
\int_G g(x)\,g(y)\,\mu(\rd g),
\]
and a reproducing kernel $K^{\rm rep}$ satisfies the following:
$K^{\rm rep}(\cdot,x)\in G$ for any $x \in D$ and 
\[
g(x)\,=\,\il g,K^{\rm rep}(\cdot,x)\ir_G\quad\mbox{for any $x\in D$ and any $g\in G$}. 
\]
Finally, in what we call the H\"older condition case, we assume that there are constants $C>0$ and $\beta\in(0,1]$ such that for any points $x$
and $y$ and any function $g$ from $G$ we have 
\[|g(x)-g(y)|\,\le\,C\,\|g\|_G\,|x-y|^\beta.
\]

Let $\omega$ denote the probability measure related to the random
variables $x_j$. To simplify the notation, we will often use
\[
Y_k\,=\,Y_k(\bsx)\,:=\,\sum_{j=1}^kx_j\,\xi_j
\quad\mbox{and}\quad
Y_\infty\,=\,Y_\infty(\bsx)\,:=\,\sum_{j=1}^\infty x_j\,\xi_j,
\]
where $\bsx=(x_j)_{j \ge 1}$. With this notation we have
\[Y_{\infty}-Y_k \,=\,\sum_{j = k+1}^{\infty} x_j\, \xi_j,
\]
a quantity that plays a crucial role in the following considerations.

\subsection{Average Case Setting} 
We assume that Fubini's theorem holds, i.e.,
\[
\bbE_\mu\, \bbE_{\omega^\bbN}\,=\,\bbE_{\omega^\bbN}\,\bbE_\mu.
\]
We would like to estimate the square
average error of approximating the expectation of $g(Y_\infty)$
by the expectation of $g(Y_k)$ over $G$ as well as the expected square 
average error of approximating
$g(Y_\infty)$ by $g(Y_k)$. The former error is given by 
\begin{eqnarray*}
e_1^{\rm trnc}(k;K_\mu^{\rm cov},\omega)&:=&
\left[\bbE_\mu\left((\bbE_{\omega^\bbN}(g(Y_\infty))
  - \bbE_{\omega^\bbN}(g(Y_k)))^2\right)\right]^{1/2}\\
&=&\left[\int_G\left(\int_{\bbR^\bbN}(g(Y_\infty(\bsx))
  -g(Y_k(\bsx)))\,\omega^\bbN(\rd \bsx)
  \right)^2\,\mu(\rd g)\right]^{1/2},
\end{eqnarray*}
and the latter by
\begin{eqnarray}\label{eqerrortwo}
e_2^{\rm trnc}(k;K_\mu^{\rm cov},\omega)&:=&\left[\bbE_{\omega^\bbN}\,\bbE_\mu
  \left(\left(g(Y_\infty)-g(Y_k)\right)^2\right)
  \right]^{1/2}\nonumber\\
&=&\left[\int_{\bbR^\bbN}\int_G\left(g(Y_\infty(\bsx))-g(Y_k(\bsx))\right)^2
  \mu(\rd g)\,  \omega^\bbN(\rd\bsx)  \right]^{1/2}.
\end{eqnarray}

\begin{remark}\label{re_errineq}
Applying the Cauchy-Schwarz inequality to the 
innermost integral in $e_1^{\rm trnc}$, it is easy to see that 
\[
e_1^{\rm trnc}(k;K_\mu^{\rm cov},\omega)\,\le\,
 e_2^{\rm trnc}(k;K_\mu^{\rm cov},\omega).
\]
Hence, in the following we will mainly concentrate on 
$e_2^{\rm trnc}(k;K_\mu^{\rm cov},\omega)$. Upper bounds on $e_2^{\rm trnc}(k;K_\mu^{\rm cov},\omega)$ also apply to $e_1^{\rm trnc}(k;K_\mu^{\rm cov},\omega)$. 
\end{remark}

\begin{proposition}
We have  
\begin{eqnarray}\label{err1}
e_1^{\rm trnc}(k;K_\mu^{\rm cov},\omega) \,=\,
 \Big[\int_{\bbR^\bbN}\int_{\bbR^\bbN}\left[
K_\mu^{\rm cov}(Y_\infty(\bsx),Y_\infty(\bsz)) -
2\, K_\mu^{\rm cov}(Y_\infty(\bsx),Y_k(\bsz)) \right.&&\nonumber\\
+\left. K_\mu^{\rm cov}(Y_k(\bsx),Y_k(\bsz))\right]
  \omega^{\bbN}(\rd\bsz)\,\omega^{\bbN}(\rd\bsx)\Big]^{1/2}.&&
\end{eqnarray}
and
\begin{eqnarray}\label{err2}
e_2^{\rm trnc}(k;K_\mu^{\rm cov},\omega) \,=\,
\Big[\int_{\bbR^\bbN} \left[K_\mu^{\rm cov}(Y_\infty(\bsx),Y_\infty(\bsx))
-2\, K_\mu^{\rm cov}(Y_\infty(\bsx),Y_k(\bsx))\right.&&\nonumber\\
+\left.
K_\mu^{\rm cov}(Y_k(\bsx),Y_k(\bsx)) \right]\omega^{\bbN}(\rd\bsx)\Big]^{1/2}.&&
\end{eqnarray}
\end{proposition}

\begin{proof}
We have 
\begin{eqnarray*}
\lefteqn{\left(e_1^{\rm trnc}(k;K_\mu^{\rm cov},\omega)\right)^2}\\
&= & \bbE_\mu\left[\left(\bbE_{\omega^\bbN}(g(Y_\infty))\right)^2
-2\,\bbE_{\omega^\bbN}(g(Y_\infty))\,\bbE_{\omega^\bbN}(g(Y_k))  +
\left(\bbE_{\omega^\bbN}(g(Y_k))\right)^2\right]\\
&=&\int_{\bbR^\bbN}\int_{\bbR^\bbN}\left[
K_\mu^{\rm cov}(Y_\infty(\bsx),Y_\infty(\bsz)) -
2\, K_\mu^{\rm cov}(Y_\infty(\bsx),Y_k(\bsz)) \right.\\
&& \qquad\qquad\qquad +\left. K_\mu^{\rm cov}(Y_k(\bsx),Y_k(\bsz))\right]
  \omega^{\bbN}(\rd\bsz)\,\omega^{\bbN}(\rd\bsx)
\end{eqnarray*}
and  
\begin{eqnarray*}
  \lefteqn{\left(e_2^{\rm trnc}(k;K_\mu^{\rm cov},\omega)\right)^2}\\
  &=&\bbE_{\omega^\bbN}
\,\bbE_\mu \left(g(Y_\infty)\,g(Y_\infty) - 2\,g(Y_\infty)\,g(Y_k)  
+ g(Y_k)\,g(Y_k)\right)\\
&=&\bbE_{\omega^\bbN}\left(K_\mu^{\rm cov}(Y_\infty,Y_\infty)
-2\, K_\mu^{\rm cov}(Y_\infty,Y_k) + K_\mu^{\rm cov}(Y_k,Y_k)\right)\\
&=&\int_{\bbR^\bbN}\left(K_\mu^{\rm cov}(Y_\infty(\bsx),Y_\infty(\bsx))
-2\,K_\mu^{\rm cov}(Y_\infty(\bsx),Y_k(\bsx))+
K_\mu^{\rm cov}(Y_k(\bsx),Y_k(\bsx))\right)\,\omega^\bbN(\rd \bsx).
\end{eqnarray*}
\end{proof}

\subsection{Worst Case Setting}
In the worst case setting,  we are interested in the worst case 
truncation error defined by 
\[
\sup_{\|g\|_G\le1}
\left[\bbE_{\omega^\bbN}\left(g(Y_\infty)-g(Y_k)\right)^2
\right]^{1/2}. 
\]
In the reproducing kernel Hilbert space setting, we will denote the
above truncation error by
\[
e_3^{\rm trnc}(k;K^{\rm rep},\omega),
\]
and in the H\"older's condition setting we will denote the error by
\[e_3^{\rm trnc}(k;G,\omega).
\]

\subsubsection{Reproducing Kernel Setting}
From the reproducing kernel property and the Cauchy-Schwarz inequality,
we have 
\begin{eqnarray}
|g(Y_\infty)-g(Y_k)|&=&\abs{\il g,K^{\rm rep}(\cdot,Y_\infty)
-K^{\rm rep}(\cdot,Y_k)\ir_G }\nonumber\\
&\le&\|g\|_G\,
\left\|K^{\rm rep}(\cdot,Y_\infty)-K^{\rm rep}(\cdot,Y_k)\right\|_G
\label{eq:sharp}
\end{eqnarray}
and
\[
\left\|K^{\rm rep}(\cdot,Y_\infty)-K^{\rm rep}(\cdot,Y_k)\right\|_G^2\,
=\,K^{\rm rep}(Y_\infty,Y_\infty)+K^{\rm rep}(Y_k,Y_k)
-2\,K^{\rm rep}(Y_\infty,Y_k).
\]
Since the inequality  \eqref{eq:sharp} is sharp, we have the following
proposition.

\begin{proposition}
We have
\begin{eqnarray}
&&e_3^{\rm trnc}(k;K^{\rm rep},\omega)\,=\,\left[\bbE_{\omega^\bbN}
\left(K^{\rm rep}(Y_\infty,Y_\infty)+K^{\rm rep}(Y_k,Y_k)
-2\,K^{\rm rep}(Y_\infty,Y_k)\right)\right]^{1/2}\label{err3}\\
  &&=\,\left[\int_{\bbR^\bbN}
  \left(K^{\rm rep}(Y_\infty(\bsx),Y_\infty(\bsx))+
  K^{\rm rep}(Y_k(\bsx),Y_k(\bsx))
  -2\,K^{\rm rep}(Y_\infty(\bsx),Y_k(\bsx))\right)\,\omega^\bbN(\rd \bsx)
  \right]^{1/2}\nonumber.
\end{eqnarray}
\end{proposition}

\begin{remark}
Observe that the dependence of $e_2^{\rm trnc}$ on the covariance kernel
$K^{\rm cov}_\mu$, see \eqref{err2} is the same as the dependence 
of $e_3^{\rm trnc}$ on the reproducing kernel $K^{\rm rep}$, see
\eqref{err3}. Moreover, any covariance kernel
is also a reproducing kernel. This is why we will estimate the truncation
errors
\begin{equation}\label{errall}
e^{\rm trnc}(k;K,\omega)\,=\,\left[
  \bbE_{\omega^\bbN}(K(Y_\infty,Y_\infty)+K(Y_k,Y_k)-2\,
  K(Y_\infty,Y_k))\right]^{1/2},
\end{equation}
for different kernels $K$ representing either covariance kernels
of probability measures $\mu$ or reproducing kernels of the spaces $G$
generated by those kernels.  
\end{remark}

\subsubsection{H\"older Condition Setting}
Due to the assumption of a H\"older condition, we immediately get
\begin{equation}\label{eq:Hoelderconstant}
e_3^{\rm trnc}(k;G,\omega)\,\le\,C\,
\left[\bbE_{\omega^\bbN}(|Y_\infty-Y_k|^{2\,\beta})\right]^{1/2}
\,=\,C\,\left[\bbE_{\omega^\bbN}\left(\sum_{j=k+1}^\infty x_j\,
  \xi_j\right)^{2\,\beta}\right]^{1/2}. 
\end{equation}

A primary example of such spaces is provided by the following.
For $p\in(1,\infty]$, let $G=G_p$ be the space of functions $g$ on
$D=[0,T]$ that are absolutely continuous with $g'\in L_p$. The norm
in the space $G_p$ is defined by
\[\|g\|_{G_p}\,=\,\left(|g(0)|^p+\|g'\|_{L_p}^p\right)^{1/p}. 
\]
Here $T$ can be any positive number or $T=\infty$. In the latter case
$D=\bbR_+=[0,\infty)$. Note that for $p=2$ the subspace of
$G_2$ with $g(0)=0$ is the reproducing kernel
Hilbert space with $K^{\rm rep}(x,y)=\min(x,y)$. It is considered in the next
section. 

Since $g(x)=g(0)+\int_0^xg'(t)\,\rd t$ for any $g\in G$, we have 
for any $x,y\in D$ with $x\ge y$ that 
\begin{eqnarray*}|g(x)-g(y)|=\left|
\int_D g'(t)\left((x-t)^0_+-(y-t)^0_+\right)\,\rd t\right|\le \|g'\|_{L_p}\,(x-y)^{1/p^*}.
\end{eqnarray*}
Here $p^*$ is the conjugate of $p$ and, in particular, $p^*=1$
if $p=\infty$. Since the H\"older inequality used above is sharp, we
conclude that functions from $G_p$ satisfy a H\"older condition
with $C=1$ and $\beta= 1/p^*$.

Of course, the same holds if the domain $D=[-T,T]$ or if it is any interval
containing $0$. Then the subspace of $G_2$ with $g(0)=0$ is the
reproducing kernel Hilbert space with kernel 
$K^{\rm rep}(x,y)=\tfrac{1}{2}(|x|+|y|-|x-y|)$.

\section{Estimates of the expectation of $|Y_\infty-Y_k|^{M}$}

We now elaborate on estimating the expectation of
$|Y_\infty-Y_k|^{M}$ with respect to $\omega^{\bbN}$. Estimates
of this particular expectation are required in order to find good
bounds on $e_3^{\rm trnc}(k;G,\omega)$ via 
\eqref{eq:Hoelderconstant}. We will see in Section~\ref{sec:appl} 
that such estimates will be also helpful in obtaining good bounds
on $e^{\rm trnc}(k;K,\omega)$ in \eqref{errall}.

In the following let 
\begin{equation}\label{def_mr}
m_r:=\bbE_\omega(|x|^r)=\int_{\bbR} |x|^r \,
\omega(\rd x), \ \ \ \mbox{for $r \in \bbN$.}
\end{equation}

First we consider the case $M=2 \beta$ for $\beta\in(0,1]$.

\begin{proposition}\label{prop:frct}
  For $\beta\,\le\,1/2$ and any $k \in \mathbb{N}_0$ we have
\begin{equation}\label{eq:fr1}
\bbE_{\omega^\bbN}\left(|Y_\infty-Y_k|^{2\,\beta}\right)\le
\left(m_1\,\sum_{j=k+1}^\infty|\xi_j|\right)^{2\,\beta}.
\end{equation}
In general, for any $\beta\in(0,1]$ and any $k \in \mathbb{N}_0$ we have 
\begin{eqnarray*}
\bbE_{\omega^\bbN}\left(|Y_\infty-Y_k|^{2\,\beta}\right)  \le
\left(\left(\bbE_\omega(x_1)\,\sum_{j=k+1}^\infty\xi_j\right)^2+
{\rm Var}_\omega(x_1)\,\sum_{j=k+1}^\infty\xi_j^2
\right)^\beta,
\end{eqnarray*}
where ${\rm Var}_\omega(x_1)=\bbE_\omega(x_1^2)-(\bbE_\omega(x_1))^2$.
Moreover, if $x_1$ is a zero-mean random variable, i.e.,
$\bbE_\omega(x_1)=0$, then
\begin{equation}\label{eq:0-mean}
  \bbE_{\omega^\bbN}\left(|Y_\infty-Y_k|^{2\,\beta}\right) \le
  \left(\bbE_\omega(x_1^2)\,\sum_{j=k+1}^\infty\xi_j^2\right)^\beta.
\end{equation}
\end{proposition}

\begin{proof}
If $2\beta\le1$ then, using H\"older's inequality with $p=1/(2\beta)$,
we get 
\begin{eqnarray*}
\bbE_{\omega^\bbN}\left(\left|\sum_{j=k+1}^\infty
x_j\,\xi_j\right|^{2\beta}\right)&\le&
\left(\bbE_{\omega^\bbN}\left|\sum_{j=k+1}^\infty
    x_j\,\xi_j\right|^{2\beta p}
\right)^{1/p}\,\left(\bbE_{\omega^\bbN} 1^{p^*}\right)^{1/p^*}\\
&=&\left(\bbE_{\omega^\bbN}\left|\sum_{j=k+1}^\infty
    x_j\,\xi_j\right|\right)^{2\beta}
\,\le\,\left(m_1\,\sum_{j=k+1}^\infty|\xi_j|\right)^{2\beta},
\end{eqnarray*}
as needed. 
In general (for $\beta\in(0,1]$) we use H\"older's inequality with 
$p=1/\beta$ and get 
\[
\bbE_{\omega^\bbN}\left(\left|\sum_{j=k+1}^\infty
x_j\,\xi_j\right|^{2\beta}\right)\,\le\,
\left(\bbE_{\omega^\bbN}\left(\sum_{j=k+1}^\infty x_j\,\xi_j\right)^2
\right)^\beta.
\]
From here the remaining results follow easily.
\end{proof}

\begin{example}
We now illustrate the bounds \eqref{eq:fr1} and \eqref{eq:0-mean} 
using uniform distribution on $[-1/2,1/2]$ and standard normal distribution on $\mathbb{R}$ 
for $\omega$, and
\[
  |\xi_j|\,\le\, j^{-a}\quad\mbox{for\ }a>1.
\]
Note that for $a>1$ we have 
\begin{equation}\label{bd_weights_ja}
  \frac1{(a-1)\,(k+1)^{a-1}}\,\le\,
  \sum_{j=k+1}^{\infty} \frac{1}{j^a}\,\le \,
  \frac1{(a-1)\,(k+1/2)^{a-1}}.
\end{equation}

Clearly, 
$m_1=1/4$ and $m_2 =1/12$ for uniform distribution, and 
$m_1=\sqrt{2/\pi}$, $m_2=1$ for the normal distribution, and in both cases
$x_1$ is zero-mean. The estimates \eqref{eq:Hoelderconstant} and \eqref{eq:fr1} together with \eqref{bd_weights_ja} give the bound
\[
  e_3^{\rm trnc}(k;G,\omega)\,\le\,
  C\, m_1^\beta\,\frac{1}{(a-1)^\beta\,(k+1/2)^{\beta(a-1)}},
\]
and \eqref{eq:Hoelderconstant} and \eqref{eq:0-mean} together with \eqref{bd_weights_ja} give
\[
  e_3^{\rm trnc}(k;G,\omega)\,\le\,
C\, m_2^{\beta/2}\,\frac{1}{(2a-1)^{\beta/2}\,(k+1/2)^{\beta(a-1/2)}},
\]
where $C$ is as in \eqref{eq:Hoelderconstant}. Note that the second bound is slightly better with respect to the order of convergence in $k$. 
\end{example}

Now we estimate the expectation of $|Y_{\infty}-Y_k|^M$ for positive integer exponents $M$.

\begin{proposition}\label{lem:powerM}
For a positive integer $M$, define 
\begin{equation}\label{eq:CM}
C(M,\omega)\,:=\,
\max\left\{\prod_{j=1}^{\ell} m_{r_j}\ :\ r_j \in \bbN,\
\sum_{j=1}^{\ell} r_j=M \ \mbox{ and }\ \ell \in \{1,2,\ldots,M\}\right\},
\end{equation}
where $m_r$ is as in \eqref{def_mr}.
Then for any $k \in \bbN_0$ we have 
\[
\bbE_{\omega^\bbN}(|Y_{\infty}-Y_k|^M) \,\le\, C(M,\omega)\,
\left(\sum_{j=k+1}^\infty |\xi_j|\right)^M.
\]
In particular, for $k=0$, we have 
$\bbE_{\omega^\bbN}(|Y_\infty^M|)\,
\le\,C(M,\omega)\,\left(\sum_{j=1}^\infty |\xi_j|\right)^M$.
\end{proposition}

\begin{proof}
We have
\begin{eqnarray*}
\bbE_{\omega^\bbN}(|Y_{\infty}-Y_k|^M) & = &  \int_{\bbR^{\bbN}}  
\left|\sum_{j=k+1}^\infty \xi_j\, x_j\right|^M \omega^\bbN(\rd\bsx)\\
& \le & \sum_{j_1 = k+1}^\infty \ldots \sum_{j_M = k+1}^\infty |\xi_{j_1} \cdots 
\xi_{j_M}| \int_{\bbR^{\bbN}} |x_{j_1}\cdots x_{j_M}| \, \omega^\bbN(\rd\bsx)\\
& \le & \left(\sum_{j=k+1}^\infty |\xi_j|\right)^M
        \max_{(j_1,\ldots,j_M)}  
\int_{\bbR^{\bbN}} |x_{j_1}\cdots x_{j_M}| \,  \omega^\bbN(\rd\bsx),
\end{eqnarray*}
where the maximum is extended over all
$(j_1,j_2,\ldots,j_M) \in \{k+1,k+2,\ldots\}^M$.

For a fixed $(j_1,j_2,\ldots,j_M) \in \{k+1,k+2,\ldots\}^M$ 
let $v_1,v_2,\ldots,v_{\ell}$ be the different $j_i$'s such that 
$v_1$ appears $r_1$ times, $v_2$ appears $r_2$ times, \ldots, 
$v_{\ell}$ appears $r_{\ell}$ times. Of course, 
$r_1+r_2+\cdots+r_{\ell}=M$ and 
$\ell=\ell(j_1,\ldots,j_M) \in \{1,2,\ldots ,M\}$. Then we have 
\[
\int_{\bbR^{\bbN}} |x_{j_1}\cdots x_{j_M}|  \omega^\bbN(\rd\bsx) 
\,=\, \int_{\bbR^{\bbN}} |x_{v_1}^{r_1}| \cdots
|x_{v_{\ell}}^{r_{\ell}}| \, 
\omega^\bbN(\rd\bsx)\,=\,m_{r_1}\cdots m_{r_{\ell}}.
\]
Hence,
\[\max_{(j_1,\ldots,j_M)}  
\int_{\bbR^{\bbN}} |x_{j_1}\cdots x_{j_M}| \,  \omega^\bbN(\rd\bsx) \le C(M,\omega)
\]
and this concludes the proof.
\end{proof}

We now provide the values of (or bounds on) $C(M,\omega)$ for a number of measures 
$\omega$.

\begin{lemma}\label{prop:Cvalues}
\begin{description}
\item{(i)} If $\omega$ is the uniform measure on $[0,1]$, then
  \[C(M,\omega)\,=\,\frac1{M+1}.
  \]
\item{(ii)} If $\omega$ is the uniform measure on $[-1/2,1/2]$, then
  \[C(M,\omega)\,=\,\frac1{2^M\,(M+1)}.
  \]
\item{(iii)} If $\omega$ is the exponential measure on $[0,\infty)$ with density
$\tfrac{1}{\lambda}\, {\rm e}^{-x/\lambda}$ for $\lambda>0$, then
  \[C(M,\omega)\,=\,\lambda^M\,M!.
  \]
\item{(iv)} If $\omega$ is the logistic measure on $\mathbb{R}$ with density
$\tfrac{1}{\lambda\,(1+{\rm e}^{-x/\lambda})^2}\,{\rm e}^{-x/\lambda}$ for $\lambda>0$, then
  \[\frac12\,\lambda^M\,M!\,<\,C(M,\omega)\,<\,2\,\lambda^M\,M!.
  \]
\item{(v)} If $\omega$ is the zero-mean Gaussian measure on $\mathbb{R}$ with density
$\tfrac{1}{\sqrt{2\,\pi\,\sigma^2}}\,{\rm e}^{-x^2/(2 \sigma^2)}$ with variance $\sigma^2>0$, then
  \[C(M,\omega)\,\le\,\sigma^{M}\,(M-1)!!,
  \]
where, for $k\in\bbN_0$,
\[
 k!!:=\prod_{j=0}^{\lceil k/2\rceil -1} (k-2j)
\]
is the double factorial of $k$. 
\end{description}
\end{lemma}

\begin{proof}
For the cases $(i)$ and $(ii)$, $m_r=1/(r+1)$ and $m_r=2^{-r}/(r+1)$,
respectively. Hence in both cases the maximum in the definition of $C(M,\omega)$
is attained for $\ell=1$.   

For the case $(iii)$, $m_r=\lambda^{r}\,r!$ and again the maximum is attained
at $\ell=1$.

For the case $(iv)$, 
\[m_r=\lambda^r\,2\,\int_0^\infty t^r\,\frac{{\rm e}^{-t}}{(1+{\rm e}^{-t})^2}\,\rd t
\,<\, \lambda^r\,2\,\int_0^\infty t^r\,{\rm e}^{-t}\,\rd t\,=\, 2\,\lambda^r\,r!,
\]
and, on the other hand, $m_r>\lambda^r\,r!/2$ which gives the bounds
for $C(M,\omega)$.

Finally, for $(v)$, 
\[
m_{2k}\,=\,\sigma^{2k}\,(2k-1)!!\quad\mbox{and}\quad
m_{2k+1}\,=\,\sigma^{2k+1}\,\sqrt{\frac2\pi}\,(2k)!!\,\le\,
\sigma^{2k+1}\,(2k)!!
\]
which yields the bound on $C(M,\omega)$. 
\end{proof}

\section{Applications}\label{sec:appl}

In this section we provide several concrete examples.

\subsection{Fractional Wiener Kernel}
Consider functions $g$ defined on $D=\bbR$ with the
(covariance or reproducing) kernel
\begin{equation}\label{defCKWiener}
K_\beta(x,y)\,=\,\frac{|x|^{2\beta}+|y|^{2\beta}-|x-y|^{2\beta}}2,\ \ \mbox{where $\beta\in(0,1)$.}
\end{equation}
The zero-mean Gaussian measure with the covariance
kernel given by $K_\beta$ is the {\em fractional Wiener} measure, see, e.g.,
\cite{R00}. Moreover, for $\beta=1/2$, it is the classical 
Wiener measure. This is why we call $K_\beta$ the fractional Wiener
kernel. 

From \eqref{errall} we obtain
\begin{eqnarray*}
\left(e^{\rm trnc}(k;K_\beta,\omega)\right)^2 &=&\bbE_{\omega^\bbN}
\left(|Y_\infty|^{2\beta} - 2\,\frac{|Y_\infty|^{2\beta}
+|Y_k|^{2\beta}-|Y_\infty-Y_k|^{2\beta}}2 + |Y_k|^{2\beta}\right) \\
&=&  \bbE_{\omega^\bbN}\left(|Y_\infty-Y_k|^{2\beta}\right). 
\end{eqnarray*}
Hence the estimates from Proposition \ref{prop:frct} apply.

\subsection{$r$-folded Wiener Kernel}\label{sec:rfoldWK}
Let $D=\bbR_+$ be the domain of functions $g$ and consider 
\begin{equation}\label{defrfoldWK}
K_r(x,y)\,=\,\int_0^\infty
\frac{(x-t)^{r-1}_+\,(y-t)^{r-1}_+}{((r-1)!)^2}\, \rd t
\end{equation}
for $r=2,3,\dots$. 
It is well known that $K_r$ is the covariance kernel of 
the $r$-folded Wiener measure. It also generates the Hilbert space 
$G_r$ 
of functions $g$ satisfying $g(0)=g^{(1)}(0)=\cdots=g^{(r-1)}(0)=0$ and the 
norm in $G_r$ is given by $\|g\|_{G_r}=\|g^{(r)}\|_{L_2(\bbR_+)}$. 

Because the domain of $g$ is $\bbR_+$, we assume that the random variables
$x_j$ take on only non-negative values and $\xi_j\ge0$.

\begin{proposition}\label{prop:est1}
Let
\begin{equation}\label{defcr}
c_r\,:=\, \left[\frac1{2r-1}+\frac{(r-1)^2}{2r-3}\right]^{1/2}\,\frac1{(r-1)!}.
\end{equation}
Suppose that $\|Y_\infty\|_{L_\infty}<\infty$, then 
\begin{equation}\label{prop:bounded}
e^{\rm trnc}(k;K_r,\omega)\,\le\,c_r\,\|Y_\infty\|_{L_\infty}^{r-3/2} \left(\left(\bbE_\omega(x_1) 
\sum_{j=k+1}^{\infty} \xi_j\right)^2+ {\rm Var}_{\omega}(x_1) \sum_{j=k+1}^{\infty} \xi_j^2 \right)^{1/2}.
\end{equation}
For the case where $\|Y_\infty\|_{L_\infty}=\infty$, but $\bbE_{\omega^\bbN}(Y_\infty^{4r-6})< \infty$, we have 
\begin{equation}\label{prop:unbounded}
e^{\rm trnc}(k;K_r,\omega)\,\le\,c_r\,\left(C(4,\omega)\,
C(4\,r-6,\omega)\right)^{1/4}\left(\sum_{j=k+1}^\infty\xi_j\right) \left(\sum_{i=1}^\infty \xi_i\right)^{r-3/2},
\end{equation}
where $C(M,\omega)$ is defined in \eqref{eq:CM}.
\end{proposition}

\begin{remark}
Note that $c_r \sim (\tfrac{r}{2})^{1/2} \tfrac{1}{(r-1)!}$ as $r \rightarrow \infty$ and $c_r \le (\tfrac{2 r}{3})^{1/2} \tfrac{1}{(r-1)!}$ for all $r \ge 2$. For simplicity we will sometimes use this bound on $c_r$ in the following.
\end{remark}

\begin{proof}
Using \eqref{errall} we obtain
\[
e^{\rm trnc}(k;K_r,\omega)\,=\, \frac{1}{(r-1)!}
\left(\int_{\bbR_+^\bbN} \int_0^\infty
\left[(Y_\infty(\bsx)-t)_+^{r-1}-(Y_k(\bsx)-t)_+^{r-1}\right]^2
\rd t\, \omega^\bbN(\rd\bsx)\right)^{1/2}.
\]
Hence we are concerned with 
\begin{eqnarray*}
 E(Y_\infty,Y_k) &:=& 
\int_0^\infty\left[(Y_\infty-t)^{r-1}_+-(Y_k-t)^{r-1}_+\right]^2\rd t= E_1+E_2,
\end{eqnarray*}
where
\[
E_1\,=\,\int_{Y_k}^{Y_\infty}(Y_\infty-t)^{2(r-1)}\rd t
\,=\, \frac{(Y_\infty-Y_k)^{2r-1}}{2r-1}
\]
and
\begin{eqnarray*}
E_2&=&\int_0^{Y_k}\left[(Y_\infty-t)^{r-1}-(Y_k-t)^{r-1}\right]^2\rd t\\
&=&\int_0^{Y_k}\left[(Y_\infty-Y_k)\,\sum_{j=0}^{r-2}(Y_\infty-t)^j\,(Y_k-t)^{r-2-j}
    \right]^2\rd t\\
&\le&(Y_\infty-Y_k)^2\,(r-1)^2\int_0^{Y_k}(Y_\infty-t)^{2r-4}\rd t\\
&\le& (Y_\infty-Y_k)^2\,\frac{(r-1)^2}{2r-3} \,Y_\infty^{2r-3}.
\end{eqnarray*}

Hence
\begin{eqnarray*}
E(Y_\infty,Y_k) & \le & \frac{(Y_\infty-Y_k)^{2r-1}}{2r-1}+ 
(Y_\infty-Y_k)^2\,\frac{(r-1)^2}{2r-3}\, Y_\infty^{2r-3}\\
& \le & (Y_\infty-Y_k)^2 Y_\infty^{2r-3}\left[\frac1{2r-1}+
\frac{(r-1)^2}{2r-3}\right].
\end{eqnarray*}
With $c_r$ as in \eqref{defcr} we get 
\begin{equation}\label{aeer}
e^{\rm trnc}(k;K_r,\omega) \le c_r \left(\int_{\bbR_+^\bbN}
  (Y_\infty-Y_k)^2\,Y_\infty^{2r-3} \,
  \omega^\bbN(\rd\bsx)\right)^{1/2}.
\end{equation}

If $\|Y_\infty\|_{L_\infty}< \infty$ we use \eqref{aeer} and Proposition~\ref{prop:frct} to obtain the desired result. 

When $\|Y_\infty\|_{L_\infty}=\infty$, but  $\bbE_{\omega^\bbN}(Y_\infty^{4r-6})< \infty$, we proceed as follows: We have
\begin{equation}\label{eq:rfold}
e^{\rm trnc}(k;K_r,\omega)\, \le\,
c_r \left(\bbE_{\omega^\bbN}(|Y_{\infty}-Y_k|^4)\right)^{1/4} \,
\left(\bbE_{\omega^\bbN}(Y_\infty^{4r-6})\right)^{1/4}.
\end{equation}
Now Proposition~\ref{lem:powerM} and \eqref{eq:rfold} yield the desired result. 
\end{proof}

As in the previous section, consider 
\[
|\xi_j|\,\le\,j^{-a}\quad\mbox{for\ }a>1,
\]
and the following two examples of $\omega$.

\begin{example}
Consider the uniform probability measure on $[0,1]$ for $\omega$. 
Then $Y_\infty(\bsx)\,\le\,\sum_{j=1}^\infty j^{-a}=\zeta(a)$ is finite and equal to the 
Riemann Zeta-Function, and \eqref{prop:bounded} together with \eqref{bd_weights_ja} yields 
\begin{eqnarray*}
e^{\rm trnc}(k;K_r,\omega) & \le &
\left(\frac{2 r}{3}\right)^{1/2}\,
\frac{\zeta(a)^{r-3/2}}{(r-1)!} \left[\frac14 \left(\sum_{j=k+1}^\infty \frac{1}{j^a}\right)^2+\frac{1}{12}\sum_{j=k+1}^\infty \frac{1}{j^{2a}}\right]^{1/2}\\
& \le &
\left(\frac{r}{6}\right)^{1/2}\,
\frac{\zeta(a)^{r-3/2}}{(r-1)!} \frac{1}{(k+1/2)^{a-1}} \left[\frac{1}{(a-1)^2} + \frac{1}{3 (2a-1)} \frac{1}{k+1/2}\right]^{1/2}\\
& \le & c_{r,a} \, \frac{1}{(k+1/2)^{a-1}},
\end{eqnarray*}
where $$c_{r,a} = \left(\frac{r}{6}\right)^{1/2}\,
\frac{\zeta(a)^{r-3/2}}{(r-1)!}  \left[\frac{1}{(a-1)^2} + \frac{2}{9 (2a-1)} \right]^{1/2}.$$
\end{example}

\begin{example}
Consider now the exponential probability measure with variance 
$\lambda>0$ for $\omega$. From Lemma \ref{prop:Cvalues} we know that 
$C(M,\omega)\,=\,\lambda^M\,M!$, and, by \eqref{prop:unbounded} and \eqref{bd_weights_ja},
\[
  e^{\rm trnc}(k;K_r,\omega)\le c_{r,\lambda} \,\frac{1}{(k+1/2)^{a-1}},
\]
where $$c_{r,\lambda}=2 r^{1/2}\lambda^{r-1/2}\,\frac{((4\,r-6)!)^{1/4}\,\zeta(a)^{r-3/2}}{(r-1)!\,(a-1)}.$$
\end{example}

\subsection{Two-Sided $r$-Folded Wiener Kernel}
Let $\bbR$ be the domain of functions $g$ and consider 
\[
K_{r,\pm}(x,y)\,=\,\left\{\begin{array}{ll}  
\int_0^\infty\frac{(|x|-t)^{r-1}_+\,(|y|-t)^{r-1}_+}
    {((r-1)!)^2} \, \rd t &\mbox{if\ } x\,y\ge 0,\\
    0 &\mbox{if\ } x\,y<0, \end{array}\right.
\]
for $r=2,3,\ldots$.

We obtain the following analogue to Proposition \ref{prop:est1}.
\begin{proposition}\label{prop:bounded2sided}
Let 
\begin{equation}\label{def:Yabs}
Y_{\infty}^{{\rm abs}}=\sum_{j=1}^{\infty} |x_j \xi_j|.
\end{equation}
Suppose that $\|Y_\infty^{{\rm abs}}\|_{L_\infty}<\infty$, then 
\[
e^{\rm trnc}(k;K_{r,\pm},\omega)\,\le\, c_r\,\|Y_\infty^{{\rm abs}}\|_{L_\infty}^{r-3/2} \left(\left(\bbE(x_1) \sum_{j=k+1}^{\infty} \xi_j\right)^2+ {\rm Var}_{\omega}(x_1) \sum_{j=k+1}^{\infty} \xi_j^2 \right)^{1/2}
.
\]
For the case where $\|Y_\infty^{{\rm abs}}\|_{L_\infty}=\infty$, but $\bbE_{\omega^\bbN}((Y_\infty^{{\rm abs}})^{4r-6})< \infty$, we have 
\[
e^{\rm trnc}(k;K_{r,\pm},\omega)\,\le\,c_r\,\left(C(4,\omega)\,
C(4\,r-6,\omega)\right)^{1/4}\left(\sum_{j=k+1}^\infty\xi_j\right) \left(\sum_{i=1}^\infty \xi_i\right)^{r-3/2},
\]
where $c_r$ is defined in \eqref{defcr} and $C(M,\omega)$ is defined in \eqref{eq:CM}.
\end{proposition}

\begin{proof}
Analogously to the proof of Proposition \ref{prop:est1}, we would like to find an upper bound on
\[
 e^{\rm trnc}(k;K_{r,\pm},\omega)\,=\,\frac{1}{(r-1)!}\left(\int_{\bbR^\bbN}E_{r,\pm}(Y_\infty,Y_k)\, \omega^{\bbN} (\rd\bsx)\right)^{1/2},
\]
where
\[
E_{r,\pm}(Y_\infty,Y_k)\,=\,((r-1)!)^2\,\left(
K_{{r,\pm}}(Y_\infty,Y_\infty)
- 2\,K_{{r,\pm}}(Y_\infty,Y_k)  +
K_{{r,\pm}}(Y_k,Y_k)\right).
\]
In the two cases when $Y_\infty$ and $Y_k$ are of the same sign,
$E_{r,\pm}(Y_\infty,Y_k)$
can be estimated as in the previous section, so we obtain
\[
E_{r,\pm}(Y_\infty,Y_k)\,\le\,
|Y_\infty-Y_k|^2 |Y_{\infty}|^{2r-3}\left[\frac1{2r-1}+
\frac{(r-1)^2}{2r-3}\right].
\]
In the case when $Y_\infty$ and $Y_k$ have different signs we have
$K_{{r,\pm}} (Y_\infty,Y_k)=0$ and
\begin{eqnarray*}
E_{r,\pm}(Y_\infty,Y_k) &=&  \int_0^{|Y_\infty|}(|Y_\infty|-t)^{2(r-1)}\rd t
+ \int_0^{|Y_k|}(|Y_k|-t)^{2(r-1)}\rd t\\
&=& \frac1{2r-1}\,\left(|Y_\infty|^{2r-1}+|Y_k|^{2r-1}\right) \le \frac1{2r-1}\,\left(|Y_\infty|+|Y_k|\right)^{2r-1}\\
& = & \frac{|Y_\infty-Y_k|^{2r-1}}{2r-1}.
\end{eqnarray*}

In any case we have  
\begin{eqnarray*}
  E_{r,\pm}(Y_\infty,Y_k) & \le & |Y_\infty-Y_k|^2 \,
  \max(|Y_{\infty}|,|Y_\infty-Y_k|)^{2r-3}\,\left[\frac1{2r-1}+
\frac{(r-1)^2}{2r-3}\right]\\
  & \le &  |Y_\infty-Y_k|^2 \, (Y_{\infty}^{{\rm abs}})^{2r-3}
  \,\left[\frac1{2r-1}+
\frac{(r-1)^2}{2r-3}\right].
\end{eqnarray*}
Hence 
\begin{eqnarray*}
  e^{\rm trnc}(k;K_{r,\pm},\omega)& \le & c_r
  \left(\int_{\bbR^\bbN} |Y_\infty-Y_k|^2 \, (Y_{\infty}^{{\rm abs}})^{2r-3}
  \, \omega^{\bbN}(\rd\bsx)\right)^{1/2}.
\end{eqnarray*}
From here the results follow in the same way as in the proof of
Proposition~\ref{prop:est1}, by noting that the proof of Proposition \ref{eq:CM} also can be used 
to bound $Y_{\infty}^{{\rm abs}}$.
\end{proof}

\begin{example}\label{ex2sideduniform}
Consider the uniform distribution on $[-1/2,1/2]$ for $\omega$. Then $\mathbb{E}_{\omega}(x_1)=0$ and 
${\rm Var}_{\omega}(x_1)=\tfrac{1}{12}$. Furthermore, $\|Y_\infty^{{\rm abs}}\|_{L_\infty}\le \tfrac12\,\sum_{j=1}^\infty|\xi_j|$.
Then we get from Proposition~\ref{prop:bounded2sided},
\[
e^{\rm trnc}(k;K_{r,\pm},\omega)\,\le\,
c_r \left(\frac12\,\sum_{j=1}^\infty|\xi_j|\right)^{r-3/2} \,\left(\frac{1}{12} \sum_{j=k+1}^\infty \xi_j^2\right)^{1/2}. 
\]
\end{example}

\begin{example}\label{ex2sidednormal}
Consider the zero mean Gaussian measure with $\sigma^2>0$ variance for
$\omega$. Then we obtain from Proposition \ref{prop:bounded2sided}
and Lemma \ref{prop:Cvalues},
\[
e^{\rm trnc}(k;K_{r,\pm},\omega)\,\le\,c_r\,
\left(\sigma^{4r-2} 3 (4r-7)!!\right)^{1/4}\left(\sum_{j=k+1}^\infty\xi_j\right) \left(\sum_{i=1}^\infty \xi_i\right)^{r-3/2}.
\]
\end{example}

\begin{remark}
  Note that it is again sufficient to assume $|\xi_j|\le j^{-a}$
  with $a>1$ to make use of the upper bounds in Examples \ref{ex2sideduniform} and \ref{ex2sidednormal}. 
\end{remark}

\subsection{Korobov Kernel}
Let $G$ be the Korobov space of functions $g$ defined on $D=[0,1]$
generated by the kernel
\[
K_r^{\rm kor}(x,y)=\sum_{h\in\bbZ} r(h)\, \ee^{2\pi\icomp h(x-y)},
\]
where $r:\bbZ\To (0,\infty)$ is a positive weight function with $r(h)=r(-h)$. 
Korobov spaces are very well studied in the field of quasi-Monte Carlo
methods, see \cite[Appendix A.1]{NW08} for an introduction. 

In \cite{NW08}, the function $r$ is such that $r(h)$ is of order
$h^{-2\alpha}$, for a nonnegative real $\alpha$. The parameter 
$\alpha$ is called the smoothness parameter of the Korobov space, 
and shows up in the norm of the space $G$. To be more precise,
the norm of $g\in G$ is
$\norm{g}_{\rm kor}=\left(\sum_{h\in\bbZ} r(h)^{-1} \abs{\widehat{g}(h)}^2\right)^{1/2}$, where $\widehat{g}(h)$ is the $h^{{\rm th}}$ 
Fourier coefficient of $g$. Hence $\alpha$ reflects the decay of the Fourier coefficients of the elements 
of $G$. Another approach, taken in \cite{KPW14}, assumes exponentially decaying $r(h)$, resulting in infinitely smooth functions
as elements of $G$. 

Due to the symmetry property of $r$, and since
\begin{eqnarray*}
\ee^{2\pi\icomp hY_\infty}-\ee^{2\pi\icomp hY_k}
  &=&\,
\ee^{\pi \icomp h(Y_\infty+Y_k)}
\left(\ee^{\pi\icomp h (Y_\infty-Y_k)}-\ee^{-\pi\icomp h (Y_\infty -Y_k)}
\right)\\
&=&\, 2\, \icomp \ee^{\pi \icomp h(Y_\infty+Y_k)}
\sin(\pi h(Y_\infty-Y_k))
\end{eqnarray*}
we obtain

\begin{eqnarray*}
K_r^{\rm kor}(Y_\infty,Y_\infty)-2\, K_r^{\rm kor}(Y_\infty,Y_k)
+ K_r^{\rm kor}(Y_k,Y_k) &=&
\,2\,\sum_{h=1}^\infty r(h)\,
\abs{\ee^{2\pi\icomp hY_\infty}-\ee^{2\pi\icomp hY_k} }^2\\
& = & \,8\,\sum_{h=1}^\infty r(h) \sin^2(\pi h (Y_\infty-Y_k))\\
& \le & \,8\,\sum_{h=1}^\infty r(h) \min\left(1\,,\,\pi^2 h^2\left(Y_\infty-Y_k\right)^2\right)\\
& \le & \, 8\,\pi^2\,|Y_\infty-Y_k|^2\,\sum_{h=1}^\infty h^2\,r(h).
\end{eqnarray*}

We assume that $r$ is such that
\[
C_r^2:=\sum_{h=1}^{\infty} h^2\,r(h)\,<\,\infty.
\]
This assumption 
is satisfied by choosing the smoothness parameter $\alpha>3/2$ in \cite{NW08}, and also
satisfied for Korobov spaces of infinitely smooth functions studied in \cite{KPW14}.
Then, according to \eqref{err3}, 
\begin{equation}\label{eqtrncerrkor}
e^{\rm trnc}(k;K^{\rm kor},\omega)\,\le\,2\sqrt{2}\,\pi\,C_r\,
(\bbE_{\omega^\bbN}(|Y_\infty-Y_k|^2))^{1/2}.
\end{equation}

The following two examples are similar to Examples \ref{ex2sideduniform} and \ref{ex2sidednormal}, and in particular can 
be used if $\abs{\xi_j }\le j^{-a}$ for $a>1$. 

\begin{example}
Consider the uniform distribution on $[-1/2,1/2]$ for $\omega$.
We can then use Proposition \ref{prop:frct} with $\beta=1$ and the fact that 
$\bbE_{\omega}(x_1^2)=1/12$, and we get from \eqref{eqtrncerrkor} and \eqref{eq:0-mean},
\[
e^{\rm trnc}(k;K_r^{{\rm kor}},\omega)\,\le\,
\sqrt{\frac{2}{3}}\,\pi\,C_r\,\left(
\sum_{j=k+1}^\infty \xi_j^2\right)^{1/2}. 
\]
\end{example}

\begin{example}
Consider the zero mean Gaussian measure with $\sigma^2>0$ variance for
$\omega$. We can then use Proposition \ref{prop:frct} with $\beta=1$ and the fact that 
$\bbE_{\omega}(x_1^2)=\sigma^2$, and we get from \eqref{eqtrncerrkor} and \eqref{eq:0-mean},
\[
e^{\rm trnc}(k;K_{r,\pm},\omega)\,\le\,
2\sqrt{2}\,\pi\,C_r\, \sigma \,\left(\sum_{j=k+1}^\infty \xi_j^2\right)^{1/2}. 
\]
\end{example}

\subsection{Hermite Kernel}
Let $G$ be a Hermite space of functions defined on $D=\bbR$
generated by the reproducing kernel
\[
K^{\rm H}_r(x,y)\,=\,\sum_{\ell=0}^\infty r(\ell)\,H_\ell(x)\,H_\ell(y),
\]
where $H_\ell$ is the $\ell^{{\rm th}}$ (normalized probabilists') Hermite polynomial
\[
H_{\ell}(x)\,=\,\frac{(-1)^{\ell}}{\sqrt{\ell!}}\,\exp(x^2/2)\,
\frac{\rd^{\ell}}{\rd x^{\ell}}\exp (-x^2/2),\quad x\in\bbR,
\]
and $r:\bbN_0\to (0,\infty)$ is a positive weight function. Integration and
function approximation over such spaces have been considered
in, e.g., \cite{DILP17,IKLP15,IL15}. 

Since $H_0\equiv1$, we have
\[
K^{\rm H}_r(Y_\infty,Y_\infty)-2\, K^{\rm H}_r(Y_\infty,Y_k) + K^{\rm H}_r(Y_k,Y_k) 
\,=\,\sum_{\ell=1}^\infty r(\ell)\, (H_{\ell} (Y_\infty)- H_{\ell} (Y_k))^2.
\]
By the mean value theorem, 
\[
\abs{H_{\ell} (Y_\infty)- H_{\ell} (Y_k)}\,=\,
\abs{H'_{\ell} (\eta_\ell)}\abs{Y_{\infty}-Y_k}
\]
for some $\eta_\ell\in I(Y_k,Y_\infty)$, where $I(Y_k,Y_\infty)=(Y_k,Y_\infty)$ if $Y_k < Y_\infty$ and $I(Y_k,Y_\infty)= (Y_\infty,Y_k)$ if $Y_k > Y_\infty$. The identity
$H'_\ell =\ell\, H_{\ell-1}$ yields
\[
\abs{H_{\ell} (Y_\infty)- H_{\ell} (Y_k)} \,=\,
\ell\abs{ H_{\ell-1} (\eta_\ell)}\abs{Y_{\infty}-Y_k}.
\]
For $\ell=1$, this yields $\abs{H_{\ell} (Y_\infty)- H_{\ell} (Y_k)}=\abs{Y_{\infty}-Y_k}$. For 
$\ell\ge 2$, we use a slightly stronger version of Cramer's bound proved in \cite{DILP17}, namely
$$
H_{\ell-1}(x)\le \min\left\{1,\frac{\sqrt{\pi}}{(\ell-1)^{1/12}}\right\}\frac{1}{\sqrt{\phi (x)}}\le \frac{c}{\ell^{1/12}}\frac{1}{\sqrt{\phi (x)}},
$$
where $\phi$ is the standard normal density function. Thus, for $\eta_\ell\in I(Y_k,Y_\infty)$ we have $$\abs{ H_{\ell-1} (\eta_\ell)} \le \frac{c}{\ell^{1/12}} \sup_{x \in I(Y_k,Y_\infty)}\sqrt[4]{2\pi} \ {\rm e}^{x^2/4}=\frac{c}{\ell^{1/12}} \sqrt[4]{2\pi} \ \max\left({\rm e}^{Y_k^2/4},{\rm e}^{Y_{\infty}^2/4}\right)\le \frac{c}{\ell^{1/12}} \sqrt[4]{2\pi} \ {\rm e}^{(Y_{\infty}^{\rm abs})^2/4},$$ where $Y_{\infty}^{\rm abs}$ is as in \eqref{def:Yabs}.

Let us now assume that
$$V:=\sum_{\ell=1}^\infty r(\ell) \ell^{11/6}<\infty.$$
We remark that this assumption is satisfied for the Hermite spaces considered in \cite{IKLP15}, and those in \cite{DILP17} if one chooses the parameter $\alpha> 17/6$ in that paper. 
Then we obtain
\begin{eqnarray*}
  \sum_{\ell=1}^\infty r(\ell)\,
  \bbE_{\omega^\bbN}\left((H_{\ell} (Y_\infty)- H_{\ell} (Y_k))^2  \right)
  &\le& c^2 \sqrt{2 \pi}\,V\, \bbE_{\omega^\bbN}\left({\rm e}^{(Y_{\infty}^{\rm abs})^2/2} (Y_{\infty}-Y_k)^2\right),
\end{eqnarray*}
for some suitably chosen $\widetilde{c}$. Hence,
\begin{equation}\label{eqtrncerrHer}
 e^{\rm trnc}(k;K^{\rm H}_r,\omega)\,\le\, c \, \left(\sqrt{2 \pi} \, V \, \bbE_{\omega^\bbN}\left({\rm e}^{(Y_{\infty}^{\rm abs})^2/2} (Y_{\infty}-Y_k)^2\right)\right)^{1/2}.
\end{equation}

Suppose that $\|{\rm e}^{(Y_{\infty}^{\rm abs})^2/2}\|_{L_{\infty}} < \infty$, then $$ e^{\rm trnc}(k;K^{\rm H}_r,\omega)\,\le\, c \, \left(\sqrt{2 \pi} \, V \, \|{\rm e}^{(Y_{\infty}^{\rm abs})^2/2}\|_{L_{\infty}} \bbE_{\omega^\bbN}\left( (Y_{\infty}-Y_k)^2\right)\right)^{1/2}.$$

\begin{example}
Consider the uniform distribution on $[-1/2,1/2]$ for $\omega$. Then we have $$\|{\rm e}^{(Y_{\infty}^{\rm abs})^2/2}\|_{L_{\infty}} \le {\rm e}^{\frac{1}{8} \left(\sum_{j=1}^{\infty} |\xi_j|\right)^2}< \infty$$ according to our standing assumption that $\sum_{j = 1}^{\infty}|\xi_j| < \infty$. We can then use Proposition \ref{prop:frct} with $\beta=1$, and the fact that 
$\bbE_{\omega}(x_1^2)=1/12$ and we get from \eqref{eq:0-mean}
\[
e^{\rm trnc}(k;K^{\rm H}_r,\omega)\,\le\, \widetilde{c} \  {\rm e}^{\frac{1}{16} \left(\sum_{j=1}^{\infty} |\xi_j|\right)^2} \left(\sum_{j=k+1}^\infty \xi_j^2\right)^{1/2}. 
\]
where $\widetilde{c}=c\ (\sqrt{2 \pi} \, V/12)^{1/2}$. This bound can be used, for example, if $\abs{\xi_j}\le j^{-a}$ with some $a>1$. In this case we have 
\[
e^{\rm trnc}(k;K^{\rm H}_r,\omega)\,\le\, \frac{\widetilde{c} \  {\rm e}^{\frac{1}{16} \zeta(a)^2} }{\sqrt{2a-1}} \frac{1}{(k+1/2)^{a-1/2}}. 
\]
\end{example}

Suppose that $\|{\rm e}^{(Y_{\infty}^{\rm abs})^2/2}\|_{L_{\infty}} = \infty$, but $\EE_{\omega^{\bbN}}({\rm e}^{(Y_{\infty}^{\rm abs})^2})< \infty$, then  $$\bbE_{\omega^\bbN}\left({\rm e}^{(Y_{\infty}^{\rm abs})^2/2} (Y_{\infty}-Y_k)^2\right) \le \EE_{\omega^{\bbN}}({\rm e}^{(Y_{\infty}^{\rm abs})^2})^{1/2}\, \bbE_{\omega^\bbN}\left((Y_{\infty}-Y_k)^4\right)^{1/2}.$$ Hence 
$$e^{\rm trnc}(k;K^{\rm H}_r,\omega)\,\le\, c \, \left(\sqrt{2 \pi} \, V \, \EE_{\omega^{\bbN}}({\rm e}^{(Y_{\infty}^{\rm abs})^2})^{1/2}\, \bbE_{\omega^\bbN}\left((Y_{\infty}-Y_k)^4\right)^{1/2}\right)^{1/2}.$$

\begin{small}
\noindent\textbf{Authors' addresses:}

\medskip

\noindent Peter Kritzer\\
Johann Radon Institute for Computational and Applied Mathematics (RICAM)\\
Austrian Academy of Sciences\\
Altenbergerstr.~69, 4040 Linz, Austria\\
E-mail: \texttt{peter.kritzer@oeaw.ac.at}

\medskip

\noindent Friedrich Pillichshammer\\
Institut f\"{u}r Finanzmathematik und Angewandte Zahlentheorie\\
Johannes Kepler Universit\"{a}t Linz\\
Altenbergerstr.~69, 4040 Linz, Austria\\
E-mail: \texttt{friedrich.pillichshammer@jku.at}

\medskip

\noindent G. W. Wasilkowski\\
Computer Science Department, University of Kentucky\\
301 David Marksbury Building\\
329 Rose Street\\
Lexington, KY 40506, USA\\
E-mail: \texttt{greg@cs.uky.edu}
\end{small}

\end{document}